\documentclass[12pt,twoside]{amsart}
\usepackage{amsmath, amsthm, amscd, amsfonts, amssymb, graphicx}
\usepackage[bookmarksnumbered, plainpages]{hyperref}

\newtheorem{thm}{Theorem}[section]

\newtheorem{lem}[thm]{Lemma}
\newtheorem{prop}[thm]{Proposition}
\newtheorem{defn}[thm]{Definition}

\numberwithin{equation}{section}

\begin{document}

\title{\bf Affine connections and Gauss-Bonnet theorems in the Heisenberg group}

\author{Yong Wang}

\thanks{{\scriptsize
\hskip -0.4 true cm \textit{2010 Mathematics Subject Classification:}
53C40; 53C42.
\newline \textit{Key words and phrases:} Schouten-van Kampen affine connections; the adapted connection, Gauss-Bonnet theorem; sub-Riemannian limit,
Heisenberg group
\newline \textit{Corresponding author:} Yong Wang}}

\maketitle

\begin{abstract}
 In this paper, we compute sub-Riemannian limits of Gaussian curvature associated to two kinds of Schouten-Van Kampen affine connections and the adapted connection for a Euclidean $C^2$-smooth surface in the Heisenberg group away from
 characteristic points and signed geodesic curvature associated to two kinds of Schouten-Van Kampen affine connections and the adapted connection
 for Euclidean $C^2$-smooth curves on surfaces. We get Gauss-Bonnet theorems associated to two kinds of Schouten-Van Kampen affine connections in the Heisenberg group.

\end{abstract}

\vskip 0.2 true cm


\pagestyle{myheadings}
\markboth{\rightline {\scriptsize Wang}}
         {\leftline{\scriptsize Gauss Bonnet theorems}}

\bigskip
\bigskip


\section{ Introduction}
 \indent In \cite{DV}, Gaussian curvature for non-horizontal surfaces in sub-Riemannian Heisenberg space $\mathbb{H}^1$ was
defined and a Gauss-Bonnet theorem was proved. In \cite{BTV},\cite{BTV1}, Balogh-Tyson-Vecchi used a Riemannnian approximation scheme to define a notion of intrinsic
Gaussian curvature for a Euclidean $C^2$-smooth surface in the Heisenberg group $\mathbb{H}^1$ away
from characteristic points, and a notion of intrinsic signed geodesic curvature for Euclidean
$C^2$-smooth curves on surfaces. These results were then used to prove a Heisenberg version of
the Gauss-Bonnet theorem. In \cite{Ve}, Veloso verified that Gausssian curvature of surfaces and normal curvature of curves in surfaces introduced by \cite{DV} and by \cite{BTV} to prove Gauss-Bonnet theorems
in Heisenberg space $\mathbb{H}^1$ were unequal and he applied the same formalism of \cite{DV} to
get the curvatures of \cite{BTV}. With the obtained formulas, it is possible to prove the Gauss-Bonnet
theorem in \cite{BTV} as a straightforward application of the Stokes theorem.
In \cite{WW1}, we proved Gauss-Bonnet theorems for the affine group and the group of rigid motions of the Minkowski plane.
In \cite{WW2}, we obtained Gauss-Bonnet theorems for BCV spaces and the twisted Heisenberg group.\\
\indent In \cite{Kl}, Klatt proved a Gauss-Bonnet theorem associated to a metric connection (see Proposition 5.2 in \cite{Kl}).
When a Riemannian manifold has a splitting tangent bundle, we can define a Schouten-Van Kampen affine connection which is a metric connection.
In \cite{Be}, \cite{Ol}, Schouten-Van Kampen affine connections on foliations and almost (para) contact manifolds were studied.
 In \cite{DV}, in order to prove a Gauss-Bonnet theorem in the Heisenberg group, the adapted connection was introduced. The adapted connection is a metric connection.
  Motivated by above works, it is interesting to study Gauss-Bonnet theorems associated to Schouten-Van Kampen affine connections and the adapted connection in the Heisenberg group. Let $\mathcal{K}^{\Sigma,\nabla^1,\infty}$ and $k^{\infty,\nabla^1,s}_{\gamma_i,\Sigma}$ be the intrinsic
  Gauss curvature associated to the second kind of Schouten-Van Kampen affine connections and the intrinsic signed geodesic curvature associated to the second kind of Schouten-van Kampen affine connections.
  Our main theorem (Theorem 3.9) in this paper is as following: (see Section 3 for related definitions)
  \begin{thm}
 Let $\Sigma\subset (\mathbb{H},g_L)$
  be a regular surface with finitely many boundary components $(\partial\Sigma)_i,$ $i\in\{1,\cdots,n\}$, given by Euclidean $C^2$-smooth regular and closed curves $\gamma_i:[0,2\pi]\rightarrow (\partial\Sigma)_i$.
 Suppose that the characteristic set $C(\Sigma)$ satisfies $\mathcal{H}^1(C(\Sigma))=0$ and that
$||\nabla_Hu||_H^{-1}$ is locally summable with respect to the Euclidean $2$-dimensional Hausdorff measure
near the characteristic set $C(\Sigma)$, then
\begin{equation}
\int_{\Sigma}\mathcal{K}^{\Sigma,\nabla^1,\infty}d\sigma_{\Sigma}+\sum_{i=1}^n\int_{\gamma_i}k^{\infty,\nabla^1,s}_{\gamma_i,\Sigma}d{s}=0.
\end{equation}
\end{thm}
\indent In Section 2, we compute sub-Riemannian limits of Gaussian curvature associated to the first kind of Schouten-Van Kampen affine connections for a Euclidean $C^2$-smooth surface in the Heisenberg group away from
 characteristic points and signed geodesic curvature associated to the first kind of Schouten-Van Kampen affine connections
 for Euclidean $C^2$-smooth curves on surfaces. We get the Gauss-Bonnet theorem associated to the first kind of Schouten-Van Kampen affine connections in the Heisenberg group. In Section 3, we compute sub-Riemannian limits of Gaussian curvature associated to the second kind of Schouten-Van Kampen affine connections for a Euclidean $C^2$-smooth surface in the Heisenberg group away from
 characteristic points and signed geodesic curvature associated to the second kind of Schouten-Van Kampen affine connections
 for Euclidean $C^2$-smooth curves on surfaces. We get the Gauss-Bonnet theorem associated to the second kind of Schouten-Van Kampen affine connections in the Heisenberg group.
 In Section 4, we compute sub-Riemannian limits of Gaussian curvature associated to the adapted connection for a Euclidean $C^2$-smooth surface in the Heisenberg group away from characteristic points and signed geodesic curvature associated to the adapted connection
 for Euclidean $C^2$-smooth curves on surfaces.


\vskip 1 true cm

\section{The Gauss-Bonnet theorem associated to the first kind of Schouten-Van Kampen affine connections in the Heisenberg group}

Firstly we introduce some notations on the Heisenberg group. Let $\mathbb{H} $ be the Heisenberg group $\mathbb{R}^3$ where the non-commutative
group law is given by
$$(a,b,c)\star (x,y,z)=(a+x,b+y,c+z-\frac{1}{2}(xb-ya)).$$
Let
\begin{equation}
X_1=\partial_{x_1}-\frac{x_2}{2}\partial_{x_3}, ~~X_2=\partial_{x_2}+\frac{x_1}{2}\partial_{x_3},~~X_3=\partial_{x_3},
\end{equation}
and ${\rm span}\{X_1,X_2,X_3\}=T\mathbb{H}.$ Let $H={\rm span}\{X_1,X_2\}$ be the horizontal distribution on $\mathbb{H}$
Let $\omega_1=dx_1,~~\omega_2=dx_2,~~\omega_3=\omega=dx_3+\frac{1}{2}(x_2dx_1-x_1dx_2).$ For the constant $L>0$, let
$g_L=\omega_1\otimes \omega_1+\omega_2\otimes \omega_2+L\omega\otimes \omega$ be the Riemannian metric on $\mathbb{H}$. Then $X_1,X_2,\widetilde{X_3}:=L^{-\frac{1}{2}}X_3$ are orthonormal basis on $T\mathbb{H}$ with respect to $g_L$. We have
\begin{equation}
[X_1,X_2]=X_3,~~[X_2,X_3]=0,~~[X_1,X_3]=0.
\end{equation}
Let $\nabla^L$ be the Levi-Civita connection on $\mathbb{H}$ with respect to $g_L$. By Lemma 2.8 in \cite{BTV}, we have
\vskip 0.5 true cm
\begin{lem}
Let $\mathbb{H}$ be the Heisenberg group, then
\begin{align}
&\nabla^L_{X_j}X_j=0,~~~1\leq j\leq 3,~~~\nabla^L_{X_1}X_2=\frac{1}{2}X_3,~~~ \nabla^L_{X_2}X_1=-\frac{1}{2}X_3,\\
&\nabla^L_{X_1}X_3=-\frac{L}{2}X_2,~~\nabla^L_{X_3}X_1=-\frac{L}{2}X_2,\notag\\
& \nabla^L_{X_2}X_3=\nabla^L_{X_3}X_2=\frac{L}{2}X_1.\notag
\end{align}
\end{lem}
Let $H^{\bot}={\rm span}\{X_3\}$ and $P:T\mathbb{H}\rightarrow H$ and $P^{\bot}:T\mathbb{H}\rightarrow H^{\bot}$ be the projections. We define the first kind of Schouten-Van Kampen affine connections in the Heisenberg group:
\begin{equation}
\nabla_XY=P\nabla^L_XPY+P^{\bot}\nabla^L_XP^{\bot}Y.
\end{equation}
By Definition 3.1 in \cite{BTV}, we have
\begin{defn}
Let $\gamma:[a,b]\rightarrow (\mathbb{H},g_L)$ be a Euclidean $C^1$-smooth curve. We say that $\gamma$ is regular if $\dot{\gamma}\neq 0$ for every $t\in [a,b].$ Moreover we say that
$\gamma(t)$ is a horizontal point of $\gamma$ if
$$\omega(\dot{\gamma}(t))=\frac{\dot{\gamma}_2(t)}{\gamma_1(t)}-\dot{\gamma}_3(t)=0.$$
\end{defn}
\vskip 0.5 true cm
\begin{defn}
Let $\gamma:[a,b]\rightarrow (\mathbb{H},g_L)$ be a Euclidean $C^2$-smooth regular curve in the Riemannian manifold $(\mathbb{H},g_L)$. The curvature
$k^{L,\nabla}_{\gamma}$ associated to $\nabla$ of $\gamma$ at $\gamma(t)$ is defined as
\begin{equation}
k^{L,\nabla}_{\gamma}:=\sqrt{\frac{||\nabla_{\dot{\gamma}}{\dot{\gamma}}||_L^2}{||\dot{\gamma}||^4_L}-\frac{\langle \nabla_{\dot{\gamma}}{\dot{\gamma}},\dot{\gamma}\rangle^2_L}{||\dot{\gamma}||^6_L}}.
\end{equation}
\end{defn}
By (2.3) and (2.4), we have
\begin{lem}
Let $\mathbb{H}$ be the Heisenberg group, then
\begin{align}
\nabla_{X_3}X_1=-\frac{L}{2}X_2,~~
\nabla_{X_3}X_2=\frac{L}{2}X_1,~~\nabla_{X_j}X_k=0,~~ {\rm for~ other}~~ X_j,X_k.
\end{align}
\end{lem}
Let $\gamma(t)=(\gamma_1(t),\gamma_2(t),\gamma_3(t))$, then
\begin{equation}
\dot{\gamma}(t)={\dot{\gamma}_1}X_1+\dot{\gamma}_2X_2+\omega(\dot{\gamma}(t))X_3.
\end{equation}
By (2.6) and (2.7), we have
\begin{align}
\nabla_{\dot{\gamma}}\dot{\gamma}=
\left[\ddot{\gamma}_1+L\omega(\dot{\gamma}(t))\frac{\dot{\gamma}_2}{2}\right]X_1+
\left[\ddot{\gamma}_2-L\omega(\dot{\gamma}(t))\frac{\dot{\gamma}_1}{2}\right]X_2
+\frac{d}{dt}(\omega(\dot{\gamma}(t)))X_3.
\end{align}
Similar to Lemma 2.4 in \cite{WW1}, we have
\vskip 0.5 true cm
\begin{lem}
Let $\gamma:[a,b]\rightarrow (\mathbb{H},g_L)$ be a Euclidean $C^2$-smooth regular curve in the Riemannian manifold $(\mathbb{H},g_L)$. Then
\begin{align}
k^{L,\nabla}_{\gamma}&=\left\{\left\{\left[
\ddot{\gamma}_1+L\omega(\dot{\gamma}(t))\frac{\dot{\gamma}_2}{2}\right]^2+\left[
\ddot{\gamma}_2-L\omega(\dot{\gamma}(t))\frac{\dot{\gamma}_1}{2}
\right]^2+L\left[\frac{d}{dt}(\omega(\dot{\gamma}(t)))\right]^2\right\}\right.\\\notag
&\cdot\left[{\dot{\gamma}_1}^2+\dot{\gamma}_2^2+L(\omega(\dot{\gamma}(t)))^2\right]^{-2}\\\notag
&-\left\{{\dot{\gamma}_1}\left[
\ddot{\gamma}_1+L\omega(\dot{\gamma}(t))\frac{\dot{\gamma}_2}{2}\right]
+\dot{\gamma}_2\left[
\ddot{\gamma}_2-L\omega(\dot{\gamma}(t))\frac{\dot{\gamma}_1}{2}
\right]
+
L\omega(\dot{\gamma}(t))\frac{d}{dt}(\omega(\dot{\gamma}(t)))\right\}^2\\\notag
&\left.\cdot\left[{\dot{\gamma}_1}^2+\dot{\gamma}_2^2+L(\omega(\dot{\gamma}(t)))^2\right]^{-3}\right\}^{\frac{1}{2}}.\\\notag
\end{align}
In particular, if $\gamma(t)$ is a horizontal point of $\gamma$,
\begin{align}
k^{L,\nabla}_{\gamma}&=\left\{\left\{
\ddot{\gamma}_1^2+
\ddot{\gamma}_2^2+L\left[\frac{d}{dt}(\omega(\dot{\gamma}(t)))\right]^2\right\}
\cdot\left[{\dot{\gamma}_1}^2+\dot{\gamma}_2^2\right]^{-2}\right.\\\notag
&\left.-\left\{{\dot{\gamma}_1}
\ddot{\gamma}_1
+\dot{\gamma}_2
\ddot{\gamma}_2
\right\}^2\cdot\left[{\dot{\gamma}_1}^2+\dot{\gamma}_2^2\right]^{-3}\right\}^{\frac{1}{2}}.\\\notag
\end{align}
\end{lem}
\vskip 0.5 true cm
\begin{defn}
Let $\gamma:[a,b]\rightarrow (\mathbb{H},g_L)$ be a Euclidean $C^2$-smooth regular curve in the Riemannian manifold $(\mathbb{H},g_L)$.
We define the intrinsic curvature associated to the connection $\nabla$, $k_{\gamma}^{\infty,\nabla}$ of $\gamma$ at $\gamma(t)$ to be
$$k_{\gamma}^{\infty,\nabla}:={\rm lim}_{L\rightarrow +\infty}k_{\gamma}^{L,\nabla},$$
if the limit exists.
\end{defn}
\indent We introduce the following notation: for continuous functions $f_1,f_2:(0,+\infty)\rightarrow \mathbb{R}$,
\begin{equation}
f_1(L)\sim f_2(L),~~as ~~L\rightarrow +\infty\Leftrightarrow {\rm lim}_{L\rightarrow +\infty}\frac{f_1(L)}{f_2(L)}=1.
\end{equation}
Similar to Lemma 2.6 in \cite{WW1}, we have
\vskip 0.5 true cm
\begin{lem}
Let $\gamma:[a,b]\rightarrow (\mathbb{H},g_L)$ be a Euclidean $C^2$-smooth regular curve in the Riemannian manifold $(\mathbb{H},g_L)$. Then
\begin{equation}
k_{\gamma}^{\infty,\nabla}=\frac{\sqrt{\dot{\gamma}_1^2+\dot{\gamma}_2^2}}{2|\omega(\dot{\gamma}(t))|},~~if ~~\omega(\dot{\gamma}(t))\neq 0,
\end{equation}
\begin{align}
k^{\infty,\nabla}_{\gamma}&=\frac{|\ddot{\gamma}_1\dot{\gamma}_2-\ddot{\gamma}_2\dot{\gamma}_1|}{(\dot{\gamma}_1^2+\dot{\gamma}_2^2)^{\frac{3}{2}}},
~~if ~~\omega(\dot{\gamma}(t))= 0 ~~and~~\frac{d}{dt}(\omega(\dot{\gamma}(t)))=0,\\\notag
\end{align}
\begin{equation}
{\rm lim}_{L\rightarrow +\infty}\frac{k_{\gamma}^{L,\nabla}}{\sqrt{L}}=\frac{|\frac{d}{dt}(\omega(\dot{\gamma}(t)))|}
{\dot{\gamma}_1^2+\dot{\gamma}_2^2},~~if ~~\omega(\dot{\gamma}(t))= 0
~~and~~\frac{d}{dt}(\omega(\dot{\gamma}(t)))\neq 0.
\end{equation}
\end{lem}
\indent We will say that a surface $\Sigma\subset(\mathbb{H},g_L)$ is regular if $\Sigma$ is a Euclidean $C^2$-smooth compact and oriented surface. In particular we will assume that there exists
a Euclidean $C^2$-smooth function $u:\mathbb{H}\rightarrow \mathbb{R}$ such that
$$\Sigma=\{(x_1,x_2,x_3)\in \mathbb{G}:u(x_1,x_2,x_3)=0\}$$
and $u_{x_1}\partial_{x_1}+u_{x_2}\partial_{x_2}+u_{x_3}\partial_{x_3}\neq 0.$ Let $\nabla_Hu=X_1(u)X_1+X_2(u)X_2.$
 A point $x\in\Sigma$ is called {\it characteristic} if $\nabla_Hu(x)=0$.
 We define the characteristic set $C(\Sigma):=\{x\in\Sigma|\nabla_Hu(x)=0\}.$
  Our computations will
be local and away from characteristic points of $\Sigma$. Let us define first
$$p:=X_1u,~~~~q:=X_2u ,~~{\rm and}~~r:=\widetilde{X}_3u.$$
We then define
\begin{align}
&l:=\sqrt{p^2+q^2},~~~~l_L:=\sqrt{p^2+q^2+r^2},~~~~\overline{p}:=\frac{p}{l},\\
&\overline{q}:=\frac{q}{l},~~~~
\overline{p_L}:=\frac{p}{l_L},~~~~\overline{q_L}:=\frac{q}{l_L},~~~~\overline{r_L}:=\frac{r}{l_L}.\notag
\end{align}
In particular, $\overline{p}^2+\overline{q}^2=1$. These functions are well defined at every non-characteristic point. Let
\begin{align}
v_L=\overline{p_L}X_1+\overline{q_L}X_2+\overline{r_L}\widetilde{X_3},~~~~e_1=\overline{q}X_1-\overline{p}X_2,~~~~
e_2=\overline{r_L}~~\overline{p}X_1+\overline{r_L}~~ \overline{q}X_2-\frac{l}{l_L}\widetilde{X_3},
\end{align}
then $v_L$ is the Riemannian unit normal vector to $\Sigma$ and $e_1,e_2$ are the orthonormal basis of $\Sigma$. On $T\Sigma$ we define a linear transformation $J_L:T\Sigma\rightarrow T\Sigma$ such that
\begin{equation}
J_L(e_1):=e_2;~~~~J_L(e_2):=-e_1.
\end{equation}
For every $U,V\in T\Sigma$, we define $\nabla^{\Sigma}_UV=\pi \nabla_UV$ where $\pi:T\mathbb{H}\rightarrow T\Sigma$ is the projection. Then $\nabla^{\Sigma}$ is the metric connection on $\Sigma$
with respect to the metric $g_L$.
By (2.8),(2.16), we have
\begin{equation}
\nabla^{\Sigma}_{\dot{\gamma}}\dot{\gamma}=\langle \nabla_{\dot{\gamma}}\dot{\gamma},e_1\rangle_Le_1+\langle \nabla_{\dot{\gamma}}\dot{\gamma},e_2\rangle_Le_2,
\end{equation}
we have
\begin{align}
\nabla^{\Sigma}_{\dot{\gamma}}\dot{\gamma}&=
\left\{\overline{q}\left[\ddot{\gamma}_1+L\omega(\dot{\gamma}(t))\frac{\dot{\gamma}_2}{2}\right]
-\overline{p}\left[\ddot{\gamma}_2-L\omega(\dot{\gamma}(t))\frac{\dot{\gamma}_1}{2}\right]
\right\}e_1\\\notag
&+\left\{\overline{r_L}~~\overline{p}\left[\ddot{\gamma}_1+L\omega(\dot{\gamma}(t))\frac{\dot{\gamma}_2}{2}\right]
\right.\\
&\left.\left.+\overline{r_L}~~\overline{q}\left[\ddot{\gamma}_2-L\omega(\dot{\gamma}(t))\frac{\dot{\gamma}_1}{2}\right]
-\frac{l}{l_L}L^{\frac{1}{2}}\frac{d}{dt}(\omega(\dot{\gamma}(t)))\right]\right\}e_2.\notag
\end{align}
\vskip 0.5 true cm
\begin{defn}
Let $\Sigma\subset(\mathbb{H},g_L)$ be a regular surface.
Let $\gamma:[a,b]\rightarrow \Sigma$ be a Euclidean $C^2$-smooth regular curve. The geodesic curvature associated to $\nabla$, $k^{L,\nabla}_{\gamma,\Sigma}$ of $\gamma$ at $\gamma(t)$ is defined as
\begin{equation}
k^{L,\nabla}_{\gamma,\Sigma}:=\sqrt{\frac{||\nabla^{\Sigma}_{\dot{\gamma}}{\dot{\gamma}}||_{\Sigma,L}^2}{||\dot{\gamma}||^4_{\Sigma,L}}-
\frac{\langle \nabla^{\Sigma}_{\dot{\gamma}}{\dot{\gamma}},\dot{\gamma}\rangle^2_{\Sigma,L}}{||\dot{\gamma}||^6_{\Sigma,L}}}.
\end{equation}
\end{defn}
\vskip 0.5 true cm
\begin{defn}
Let $\Sigma\subset(\mathbb{H},g_L)$ be a regular surface. Let $\gamma:[a,b]\rightarrow \Sigma$ be a Euclidean $C^2$-smooth regular curve.
We define the intrinsic geodesic curvature associated to $\nabla$, $k_{\gamma,\Sigma}^{\infty,\nabla}$ of $\gamma$ at $\gamma(t)$ to be
$$k_{\gamma,\Sigma}^{\infty,\nabla}:={\rm lim}_{L\rightarrow +\infty}k_{\gamma,\Sigma}^{L,\nabla},$$
if the limit exists.
\end{defn}
\vskip 0.5 true cm
Similar to Lemma 3.3 in \cite{WW1}, we have
\begin{lem}
Let $\Sigma\subset(\mathbb{H},g_L)$ be a regular surface.
Let $\gamma:[a,b]\rightarrow \Sigma$ be a Euclidean $C^2$-smooth regular curve. Then
\begin{equation}
k_{\gamma,\Sigma}^{\infty,\nabla}=\frac{|\overline{p}\dot{\gamma}_1+\overline{q}\dot{\gamma}_2|}{2|\omega(\dot{\gamma}(t))|},~~if ~~\omega(\dot{\gamma}(t))\neq 0,
\end{equation}
$$k^{\infty,\nabla}_{\gamma,\Sigma}=0
~~if ~~\omega(\dot{\gamma}(t))= 0, ~~and~~\frac{d}{dt}(\omega(\dot{\gamma}(t)))=0,$$
\begin{equation}
{\rm lim}_{L\rightarrow +\infty}\frac{k_{\gamma,\Sigma}^{L,\nabla}}{\sqrt{L}}=\frac{|\frac{d}{dt}(\omega(\dot{\gamma}(t)))|}
{\left(\overline{q}{\dot{\gamma}_1}-\overline{p}\dot{\gamma}_2\right)^2},~~if ~~\omega(\dot{\gamma}(t))= 0
~~and~~\frac{d}{dt}(\omega(\dot{\gamma}(t)))\neq 0.
\end{equation}
\end{lem}
\vskip 0.5 true cm
\begin{defn}
Let $\Sigma\subset(\mathbb{H},g_L)$ be a regular surface.
Let $\gamma:[a,b]\rightarrow \Sigma$ be a Euclidean $C^2$-smooth regular curve. The signed geodesic curvature associated to $\nabla$, $k^{L,\nabla,s}_{\gamma,\Sigma}$ of $\gamma$ at $\gamma(t)$ is defined as
\begin{equation}
k^{L,\nabla,s}_{\gamma,\Sigma}:=\frac{\langle \nabla^{\Sigma}_{\dot{\gamma}}{\dot{\gamma}},J_L(\dot{\gamma})\rangle_{\Sigma,L}}{||\dot{\gamma}||^3_{\Sigma,L}}.
\end{equation}
\end{defn}
\vskip 0.5 true cm
\begin{defn}
Let $\Sigma\subset(\mathbb{H},g_L)$ be a regular surface. Let $\gamma:[a,b]\rightarrow \Sigma$ be a Euclidean $C^2$-smooth regular curve.
We define the intrinsic geodesic curvature associated to $\nabla$, $k_{\gamma,\Sigma}^{\infty,\nabla,s}$ of $\gamma$ at the non-characteristic point $\gamma(t)$ to be
$$k_{\gamma,\Sigma}^{\infty,\nabla,s}:={\rm lim}_{L\rightarrow +\infty}k_{\gamma,\Sigma}^{L,\nabla,s},$$
if the limit exists.
\end{defn}
\vskip 0.5 true cm
Similar to Lemma 3.6 in \cite{WW1}, we have
\begin{lem}
Let $\Sigma\subset(\mathbb{H},g_L)$ be a regular surface.
Let $\gamma:[a,b]\rightarrow \Sigma$ be a Euclidean $C^2$-smooth regular curve. Then
\begin{equation}
k_{\gamma,\Sigma}^{\infty,\nabla,s}=\frac{\overline{p}\dot{\gamma}_1+\overline{q}\dot{\gamma}_2}{2|\omega(\dot{\gamma}(t))|},~~if ~~\omega(\dot{\gamma}(t))\neq 0,
\end{equation}
$$k^{\infty,\nabla,s}_{\gamma,\Sigma}=0
~~if ~~\omega(\dot{\gamma}(t))= 0, ~~and~~\frac{d}{dt}(\omega(\dot{\gamma}(t)))=0,$$
\begin{equation}
{\rm lim}_{L\rightarrow +\infty}\frac{k_{\gamma,\Sigma}^{L,\nabla,s}}{\sqrt{L}}=
\frac{(-\overline{q}{\dot{\gamma}_1}
+\overline{p}\dot{\gamma}_2)\frac{d}{dt}(\omega(\dot{\gamma}(t)))}{|\overline{q}{\dot{\gamma}_1}-\overline{p}\dot{\gamma}_2|^3},~~if ~~\omega(\dot{\gamma}(t))= 0
~~and~~\frac{d}{dt}(\omega(\dot{\gamma}(t)))\neq 0.
\end{equation}
\end{lem}
\indent In the following, we compute the sub-Riemannian limit of the Gaussian curvature associated to $\nabla$ of surfaces in the Heisenberg group. We define the {\it second fundamental form associated to $\nabla$}, $II^{\nabla,L}$ of the
embedding of $\Sigma$ into $(\mathbb{H},g_L)$:
\begin{equation}
II^{\nabla,L}=\left(
  \begin{array}{cc}
   \langle \nabla_{e_1}v_L,e_1)\rangle_{L},
    & \langle \nabla_{e_1}v_L,e_2)\rangle_{L} \\
   \langle \nabla_{e_2}v_L,e_1)\rangle_{L},
    & \langle \nabla_{e_2}v_L,e_2)\rangle_{L} \\
  \end{array}
\right).
\end{equation}
Similarly to Theorem 4.3 in \cite{CDPT}, we have
\vskip 0.5 true cm
\begin{thm} The second fundamental form $II^{\nabla,L}$ of the
embedding of $\Sigma$ into $(\mathbb{H},g_L)$ is given by
\begin{equation}
II^{\nabla,L}=\left(
  \begin{array}{cc}
   h_{11}, & h_{12}\\
   h_{21}, & h_{22}\\
  \end{array}
\right),
\end{equation}
where $$h_{11}= \frac{l}{l_L}[X_1(\overline{p})+X_2(\overline{q})],~~
    h_{12}=-\frac{l_L}{l}\langle e_1,\nabla_H(\overline{r_L})\rangle_L,$$
    $$h_{21}=-\frac{l_L}{l}\langle e_1,\nabla_H(\overline{r_L})\rangle_L-\frac{\sqrt{L}}{2}-\frac{\sqrt{L}}{2}r^2_L,$$
    $$h_{22}=-\frac{l^2}{l_L^2}\langle e_2,\nabla_H(\frac{r}{l})\rangle_L+\widetilde{X_3}(\overline{r_L}).$$
\end{thm}
\begin{proof} By Theorem 4.3 in \cite{CDPT} and Lemma 2.4, we have
\begin{align}
&\langle \nabla_{e_1}v_L,e_1\rangle_{L}=\langle \nabla^L_{e_1}v_L,e_1\rangle_{L},~~
\langle \nabla_{e_1}v_L,e_2\rangle_{L}=\langle \nabla^L_{e_1}v_L,e_2\rangle_{L}+\frac{\sqrt{L}}{2},\\\notag
&\langle \nabla_{e_2}v_L,e_1\rangle_{L}=\langle \nabla^L_{e_2}v_L,e_1\rangle_{L}-\frac{\sqrt{L}}{2}r^2_L,~~
\langle \nabla_{e_2}v_L,e_2\rangle_{L}=\langle \nabla^L_{e_2}v_L,e_2\rangle_{L},\\\notag
\end{align}
By Theorem 4.3 in \cite{CDPT} and (2.28), we get this theorem.
\end{proof}
\vskip 0.5 true cm
\indent The mean curvature associated to $\nabla$, $\mathcal{H}_{\nabla,L}$ of $\Sigma$ is defined by
$$\mathcal{H}_{\nabla,L}:={\rm tr}(II^{\nabla,L}).$$
Define the curvature of a connection $\nabla$ by
\begin{equation}
R(X,Y)Z=\nabla_X\nabla_Y-\nabla_Y\nabla_X-\nabla_{[X,Y]}.
\end{equation}
Let
\begin{equation}
\mathcal{K}^{\Sigma,\nabla}(e_1,e_2)=-\langle R^{\Sigma}(e_1,e_2)e_1,e_2\rangle_{\Sigma,L},~~~~\mathcal{K}^{\nabla}(e_1,e_2)=-\langle R(e_1,e_2)e_1,e_2\rangle_L.
\end{equation}
By the Gauss equation (in fact the Gauss equation holds for any metric connections), we have
\begin{equation}
\mathcal{K}^{\Sigma,\nabla}(e_1,e_2)=\mathcal{K}^{\nabla}(e_1,e_2)+{\rm det}(II^{\nabla,L}).
\end{equation}
Similar to Proposition 3.8 in \cite{WW1}, we have
\vskip 0.5 true cm
\begin{prop} Away from characteristic points, the horizontal mean curvature associated to $\nabla$, $\mathcal{H}_{\nabla,\infty}$ of $\Sigma\subset\mathbb{H}$ is given by
\begin{equation}
\mathcal{H}_{\nabla,\infty}={\rm lim}_{L\rightarrow +\infty}\mathcal{H}_{\nabla,L}=X_1(\overline{p})+X_2(\overline{q}).
\end{equation}
\end{prop}
\vskip 0.5 true cm
By Lemma 2.4 and (2.29), we have
\vskip 0.5 true cm
\begin{lem}
Let $\mathbb{H}$ be the Heisenberg group, then
\begin{align}
R(X_1,X_2)X_1=\frac{L}{2}X_2,~~~ R(X_1,X_2)X_2=-\frac{L}{2}X_1,
~~~ R(X_i,X_j)X_k=0,~~{\rm for~ other}~ i,j,k.\notag
\end{align}
\end{lem}
\vskip 0.5 true cm
\begin{prop} Away from characteristic points, we have
\begin{equation}
\mathcal{K}^{\Sigma,\nabla}(e_1,e_2)\rightarrow \mathcal{K}^{\Sigma,\nabla,\infty}+O(\frac{1}{\sqrt{L}}),~~{\rm as}~~L\rightarrow +\infty,
\end{equation}
where
\begin{equation}
\mathcal{K}^{\Sigma,\nabla,\infty}:=-\frac{1}{2}\langle e_1,\nabla_H(\frac{X_3u}{|\nabla_Hu|})\rangle
-\frac{(X_3u)^2}{2(p^2+q^2)}.
\end{equation}
\end{prop}
\begin{proof} By Lemma 2.16 and similar to (3.33) and (3.34) in \cite{WW1}, we have
\begin{align}
\mathcal{K}^{\nabla}(e_1,e_2)=-\frac{L}{2}\overline{r_L}^2.
\end{align}
By Theorem 2.14, (2.31) and (2.35), similar to Proposition 3.10 in \cite{WW1}, we can obtain this proposition.
\end{proof}
Let us first consider the case of a regular curve $\gamma:[a,b]\rightarrow (\mathbb{H},g_L)$. We define the Riemannian length measure
$ds_L=||\dot{\gamma}||_Ldt.$
By \cite{BTV},
we have
\begin{equation}
\frac{1}{\sqrt{L}}ds_L\rightarrow ds:=|\omega(\dot{\gamma}(t))|dt ~~{\rm as}~~L\rightarrow +\infty.
\end{equation}
\begin{equation}
\frac{1}{\sqrt{L}}e^*_1\wedge e^*_2\rightarrow d\sigma_\Sigma:=\overline{p}\omega_2\wedge \omega_3-\overline{q}\omega_1\wedge \omega_3 ~~{\rm as}~~L\rightarrow +\infty,
\end{equation}
where $e^*_1,e^*_2$ are the dual basis of $e_1,e_2$.
We recall the local Gauss-Bonnet theorem for the metric connection(see Proposition 5.2 in \cite{Kl}).
\begin{thm}
 Let $\Sigma$
  be an oriented compact two-dimensional manifold with many boundary components $(\partial\Sigma)_i,$ $i\in\{1,\cdots,n\}$, given by Euclidean $C^2$-smooth regular and closed curves $\gamma_i:[0,2\pi]\rightarrow (\partial\Sigma)_i$.
 Let $\nabla$ be a metric connection and $\mathcal{K}^{\nabla}$ be the Gauss curvature associated to $\nabla$ and $k^{s,\nabla}_{\gamma_i}$
 be the signed geodesic curvature associated to $\nabla$, then
\begin{equation}
\int_{\Sigma}\mathcal{K}^{\nabla}d\sigma_{\Sigma}+\sum_{i=1}^n\int_{\gamma_i}k^{s,\nabla}_{\gamma_i}d{s}=2\pi \chi(M).
\end{equation}
\end{thm}
By Lemma 2.13 and Proposition 2.17 and Theorem 2.18, similar to the proof of Theorem 1.1 in \cite{BTV}, we have

 \begin{thm}
 Let $\Sigma\subset (\mathbb{H},g_L)$
  be a regular surface with finitely many boundary components $(\partial\Sigma)_i,$ $i\in\{1,\cdots,n\}$, given by Euclidean $C^2$-smooth regular and closed curves $\gamma_i:[0,2\pi]\rightarrow (\partial\Sigma)_i$.
 Suppose that the characteristic set $C(\Sigma)$ satisfies $\mathcal{H}^1(C(\Sigma))=0$ and that
$||\nabla_Hu||_H^{-1}$ is locally summable with respect to the Euclidean $2$-dimensional Hausdorff measure
near the characteristic set $C(\Sigma)$, then
\begin{equation}
\int_{\Sigma}\mathcal{K}^{\Sigma,\nabla,\infty}d\sigma_{\Sigma}+\sum_{i=1}^n\int_{\gamma_i}k^{\infty,\nabla,s}_{\gamma_i,\Sigma}d{s}=0.
\end{equation}
\end{thm}

By Lemma 2.13 and (2.34), we note that Theorem 2.19 is the same as Theorem 1.1 in \cite{BTV} up to the scaler $\frac{1}{2}$.

\section{The Gauss-Bonnet theorem associated to the second kind of Schouten-van Kampen affine connections in the Heisenberg group}

Let $\overline{H}={\rm span}\{X_2,X_3\}$ and $\overline{H}^{\bot}={\rm span}\{X_1\}$ and $\overline{P}:T\mathbb{H}\rightarrow \overline{H}$ and $\overline{P}^{\bot}:T\mathbb{H}\rightarrow \overline{H}^{\bot}$ be the projections. We define the second kind of Schouten-van Kampen affine connections in the Heisenberg group:
\begin{equation}
\nabla^1_XY=\overline{P}\nabla^L_X\overline{P}Y+\overline{P}^{\bot}\nabla^L_X\overline{P}^{\bot}Y.
\end{equation}
By Lemma 2.1 and (3.1), we have
\begin{lem}
Let $\mathbb{H}$ be the Heisenberg group, then
\begin{align}
\nabla^1_{X_1}X_2=\frac{1}{2}X_3,~~
\nabla^1_{X_1}X_3=-\frac{L}{2}X_2,~~\nabla^1_{X_j}X_k=0,~~ {\rm for~ other}~~ X_j,X_k.
\end{align}
\end{lem}
Similar to the definition 2,3, we can define the curvature
$k^{L,\nabla^1}_{\gamma}$ associated to $\nabla^1$ of $\gamma$ at $\gamma(t)$.
By Lemma 3.1 and (2.7), we have
\begin{align}
\nabla^1_{\dot{\gamma}}\dot{\gamma}=
\ddot{\gamma}_1X_1+
\left[\ddot{\gamma}_2-L\omega(\dot{\gamma}(t))\frac{\dot{\gamma}_1}{2}\right]X_2
+\left[\frac{d}{dt}(\omega(\dot{\gamma}(t)))+\frac{1}{2}\dot{\gamma}_1\dot{\gamma}_2\right]X_3.
\end{align}
Similar to Lemma 2.5, we have
\vskip 0.5 true cm
\begin{lem}
Let $\gamma:[a,b]\rightarrow (\mathbb{H},g_L)$ be a Euclidean $C^2$-smooth regular curve in the Riemannian manifold $(\mathbb{H},g_L)$. Then
\begin{align}
k^{L,\nabla^1}_{\gamma}&=\left\{\left\{
\ddot{\gamma}_1^2+\left[
\ddot{\gamma}_2-L\omega(\dot{\gamma}(t))\frac{\dot{\gamma}_1}{2}
\right]^2+L\left[\frac{d}{dt}(\omega(\dot{\gamma}(t)))+\frac{1}{2}\dot{\gamma}_1\dot{\gamma}_2\right]^2\right\}\right.\\\notag
&\cdot\left[{\dot{\gamma}_1}^2+\dot{\gamma}_2^2+L(\omega(\dot{\gamma}(t)))^2\right]^{-2}\\\notag
&-\left\{{\dot{\gamma}_1}
\ddot{\gamma}_1
+\dot{\gamma}_2\left[
\ddot{\gamma}_2-L\omega(\dot{\gamma}(t))\frac{\dot{\gamma}_1}{2}
\right]
+
L\omega(\dot{\gamma}(t))\left[\frac{d}{dt}(\omega(\dot{\gamma}(t)))+\frac{1}{2}\dot{\gamma}_1\dot{\gamma}_2\right]\right\}^2\\\notag
&\left.\cdot\left[{\dot{\gamma}_1}^2+\dot{\gamma}_2^2+L(\omega(\dot{\gamma}(t)))^2\right]^{-3}\right\}^{\frac{1}{2}}.\\\notag
\end{align}
In particular, if $\gamma(t)$ is a horizontal point of $\gamma$,
\begin{align}
k^{L,\nabla^1}_{\gamma}&=\left\{\left\{
\ddot{\gamma}_1^2+
\ddot{\gamma}_2^2+L\left[\frac{d}{dt}(\omega(\dot{\gamma}(t)))+\frac{1}{2}\dot{\gamma}_1\dot{\gamma}_2\right]^2\right\}
\cdot\left[{\dot{\gamma}_1}^2+\dot{\gamma}_2^2\right]^{-2}\right.\\\notag
&\left.-\left\{{\dot{\gamma}_1}
\ddot{\gamma}_1
+\dot{\gamma}_2
\ddot{\gamma}_2
\right\}^2\cdot\left[{\dot{\gamma}_1}^2+\dot{\gamma}_2^2\right]^{-3}\right\}^{\frac{1}{2}}.\\\notag
\end{align}
\end{lem}
\vskip 0.5 true cm
\indent Similar to the definition 2.6, we can define the intrinsic curvature associated to the connection $\nabla^1$, $k_{\gamma}^{\infty,\nabla^1}$ of $\gamma$ at $\gamma(t)$. Similar to the lemma 2.7, we have
\vskip 0.5 true cm
\begin{lem}
Let $\gamma:[a,b]\rightarrow (\mathbb{H},g_L)$ be a Euclidean $C^2$-smooth regular curve in the Riemannian manifold $(\mathbb{H},g_L)$. Then
\begin{equation}
k_{\gamma}^{\infty,\nabla^1}=\frac{|\dot{\gamma}_1|}{2|\omega(\dot{\gamma}(t))|},~~if ~~\omega(\dot{\gamma}(t))\neq 0,
\end{equation}
\begin{align}
k^{\infty,\nabla^1}_{\gamma}&=\frac{|\ddot{\gamma}_1\dot{\gamma}_2-\ddot{\gamma}_2\dot{\gamma}_1|}{(\dot{\gamma}_1^2+\dot{\gamma}_2^2)^{\frac{3}{2}}},
~~if ~~\omega(\dot{\gamma}(t))= 0 ~~and~~\frac{d}{dt}(\omega(\dot{\gamma}(t)))+\frac{1}{2}\dot{\gamma}_1\dot{\gamma}_2=0,\\\notag
\end{align}
\begin{equation}
{\rm lim}_{L\rightarrow +\infty}\frac{k_{\gamma}^{L,\nabla^1}}{\sqrt{L}}=\frac{|\frac{d}{dt}(\omega(\dot{\gamma}(t)))+\frac{1}{2}\dot{\gamma}_1\dot{\gamma}_2|}
{\dot{\gamma}_1^2+\dot{\gamma}_2^2},~~if ~~\omega(\dot{\gamma}(t))= 0
~~and~~\frac{d}{dt}(\omega(\dot{\gamma}(t)))+\frac{1}{2}\dot{\gamma}_1\dot{\gamma}_2\neq 0.
\end{equation}
\end{lem}
For every $U,V\in T\Sigma$, we define $\nabla^{1,\Sigma}_UV={\pi} \nabla^1_UV$ where $\pi:T\mathbb{H}\rightarrow T\Sigma$ is the projection.
Similar to (2.18), we have
\begin{equation}
\nabla^{1,\Sigma}_{\dot{\gamma}}\dot{\gamma}=\langle \nabla^1_{\dot{\gamma}}\dot{\gamma},e_1\rangle_Le_1+\langle \nabla_{\dot{\gamma}}^1\dot{\gamma},e_2\rangle_Le_2,
\end{equation}
and
\begin{align}
\nabla^{1,\Sigma}_{\dot{\gamma}}\dot{\gamma}&=
\left\{\overline{q}\ddot{\gamma}_1
-\overline{p}\left[\ddot{\gamma}_2-L\omega(\dot{\gamma}(t))\frac{\dot{\gamma}_1}{2}\right]
\right\}e_1\\\notag
&+\left\{\overline{r_L}~~\overline{p}\ddot{\gamma}_1
+\overline{r_L}~~\overline{q}\left[\ddot{\gamma}_2-L\omega(\dot{\gamma}(t))\frac{\dot{\gamma}_1}{2}\right]
-\frac{l}{l_L}L^{\frac{1}{2}}\left[\frac{d}{dt}(\omega(\dot{\gamma}(t)))+\frac{1}{2}\dot{\gamma}_1\dot{\gamma}_2\right]\right\}e_2.\notag
\end{align}
Similar to Definitions 2.8 and 2.9, we can define
the geodesic curvature associated to $\nabla^1$, $k^{L,\nabla^1}_{\gamma,\Sigma}$ of $\gamma$ at $\gamma(t)$
and the intrinsic geodesic curvature associated to $\nabla^1$, $k_{\gamma,\Sigma}^{\infty,\nabla^1}$ of $\gamma$ at $\gamma(t)$.
Similar to Lemma 2.10, we have
\begin{lem}
Let $\Sigma\subset(\mathbb{H},g_L)$ be a regular surface.
Let $\gamma:[a,b]\rightarrow \Sigma$ be a Euclidean $C^2$-smooth regular curve. Then
\begin{equation}
k_{\gamma,\Sigma}^{\infty,\nabla^1}=\frac{|\overline{p}\dot{\gamma}_1|}{2|\omega(\dot{\gamma}(t))|},~~if ~~\omega(\dot{\gamma}(t))\neq 0,
\end{equation}
$$k^{\infty,\nabla^1}_{\gamma,\Sigma}=0
~~if ~~\omega(\dot{\gamma}(t))= 0, ~~and~~\frac{d}{dt}(\omega(\dot{\gamma}(t)))+\frac{1}{2}\dot{\gamma}_1\dot{\gamma}_2=0,$$
\begin{equation}
{\rm lim}_{L\rightarrow +\infty}\frac{k_{\gamma,\Sigma}^{L,\nabla^1}}{\sqrt{L}}=\frac{|\frac{d}{dt}(\omega(\dot{\gamma}(t)))+\frac{1}{2}\dot{\gamma}_1\dot{\gamma}_2|}
{\left(\overline{q}{\dot{\gamma}_1}-\overline{p}\dot{\gamma}_2\right)^2},~~if ~~\omega(\dot{\gamma}(t))= 0
~~and~~\frac{d}{dt}(\omega(\dot{\gamma}(t)))+\frac{1}{2}\dot{\gamma}_1\dot{\gamma}_2\neq 0.
\end{equation}
\end{lem}
\indent Similar to the definitions 2.11 and 2.12, we can define
the signed geodesic curvature associated to $\nabla^1$, $k^{L,\nabla^1,s}_{\gamma,\Sigma}$ of $\gamma$ at $\gamma(t)$ and
the intrinsic geodesic curvature associated to $\nabla^1$, $k_{\gamma,\Sigma}^{\infty,\nabla^1,s}$ of $\gamma$.
Similar to Lemma 2.13, we have
\begin{lem}
Let $\Sigma\subset(\mathbb{H},g_L)$ be a regular surface.
Let $\gamma:[a,b]\rightarrow \Sigma$ be a Euclidean $C^2$-smooth regular curve. Then
\begin{equation}
k_{\gamma,\Sigma}^{\infty,\nabla^1,s}=\frac{\overline{p}\dot{\gamma}_1}{2|\omega(\dot{\gamma}(t))|},~~if ~~\omega(\dot{\gamma}(t))\neq 0,
\end{equation}
$$k^{\infty,\nabla^1,s}_{\gamma,\Sigma}=0
~~if ~~\omega(\dot{\gamma}(t))= 0, ~~and~~\frac{d}{dt}(\omega(\dot{\gamma}(t)))+\frac{1}{2}\dot{\gamma}_1\dot{\gamma}_2=0,$$
\begin{equation}
{\rm lim}_{L\rightarrow +\infty}\frac{k_{\gamma,\Sigma}^{L,\nabla^1,s}}{\sqrt{L}}=
\frac{(-\overline{q}{\dot{\gamma}_1}
+\overline{p}\dot{\gamma}_2)\left[\frac{d}{dt}(\omega(\dot{\gamma}(t)))+\frac{1}{2}\dot{\gamma}_1\dot{\gamma}_2\right]}
{|\overline{q}{\dot{\gamma}_1}-\overline{p}\dot{\gamma}_2|^3},
\end{equation}
$$~~if ~~\omega(\dot{\gamma}(t))= 0
~~and~~\frac{d}{dt}(\omega(\dot{\gamma}(t)))+\frac{1}{2}\dot{\gamma}_1\dot{\gamma}_2\neq 0.
$$
\end{lem}
\indent Similar to (2.26), we can define the {\it second fundamental form associated to $\nabla^1$}, $II^{\nabla^1,L}$ of the
embedding of $\Sigma$ into $(\mathbb{H},g_L)$.
Similarly to Theorem 2.14, we have
\vskip 0.5 true cm
\begin{thm} The second fundamental form $II^{\nabla^1,L}$ of the
embedding of $\Sigma$ into $(\mathbb{H},g_L)$ is given by
\begin{equation}
II^{\nabla^1,L}=\left(
  \begin{array}{cc}
   h^1_{11}, & h^1_{12}\\
   h^1_{21}, & h^1_{22}\\
  \end{array}
\right),
\end{equation}
where $$h^1_{11}= \frac{l}{l_L}[X_1(\overline{p})+X_2(\overline{q})]+\frac{\sqrt{L}\overline{p}~\overline{q}~\overline{r_L}}{2},$$
    $$h^1_{12}=-\frac{l_L}{l}\langle e_1,\nabla_H(\overline{r_L})\rangle_L-\frac{1}{2}\overline{r_L}^2~\overline{q}^2\sqrt{L}
    -\frac{l}{2l_L}\overline{q}~\overline{q_L}\sqrt{L}
    ,$$
    $$h^1_{21}=-\frac{l_L}{l}\langle e_1,\nabla_H(\overline{r_L})\rangle_L-\frac{\sqrt{L}}{2}+\frac{\sqrt{L}}{2}\frac{l^2}{l_L^2}
    -\frac{\sqrt{L}}{2}\overline{r_L}^2\overline{q}^2,$$
   $$h^1_{22}=-\frac{l^2}{l_L^2}\langle e_2,\nabla_H(\frac{r}{l})\rangle_L+\widetilde{X_3}(\overline{r_L})
    -\frac{\sqrt{L}}{2}\frac{l}{l_L}\overline{p}~\overline{q_L}~\overline{r_L}
     -\frac{\sqrt{L}}{2}\overline{p}~\overline{q}~\overline{r_L}^3 .$$
\end{thm}
Similar to (2.29) and (2.30), we can define $R^1(X,Y)Z$, $\mathcal{K}^{\Sigma,\nabla^1}(e_1,e_2)$ and $\mathcal{K}^{\nabla^1}(e_1,e_2)$
(2.31) is correct for $\nabla^1$. Similar to Proposition 2.15, we have
\vskip 0.5 true cm
\begin{prop} Away from characteristic points, the horizontal mean curvature associated to $\nabla^1$, $\mathcal{H}_{\nabla^1,\infty}$ of $\Sigma\subset\mathbb{H}$ is given by
\begin{equation}
\mathcal{H}_{\nabla^1,\infty}={\rm lim}_{L\rightarrow +\infty}\mathcal{H}_{\nabla^1,L}=X_1(\overline{p})+X_2(\overline{q}).
\end{equation}
\end{prop}
\vskip 0.5 true cm
By Lemma 3.1, we have
$R^1(X_i,X_j)X_k=0$ for any $i,j,k$.
Similar to Proposition 2.17, we have
\begin{prop} Away from characteristic points, we have
\begin{equation}
\mathcal{K}^{\Sigma,\nabla^1}(e_1,e_2)\rightarrow \mathcal{K}^{\Sigma,\nabla^1,\infty}+O(\frac{1}{\sqrt{L}}),~~{\rm as}~~L\rightarrow +\infty,
\end{equation}
where
\begin{equation}
\mathcal{K}^{\Sigma,\nabla^1,\infty}:=-\frac{\overline{p}~\overline{q}(X_3(u))}{2\sqrt{p^2+q^2}}[X_1(\overline{p})+X_2(\overline{q})]
-\frac{\overline{q}^2}{2}\left[\langle e_1,\nabla_H(\frac{X_3u}{|\nabla_Hu|})\rangle
+\frac{(X_3u)^2}{p^2+q^2}\right].
\end{equation}
\end{prop}
By Lemma 3.5 and Proposition 3.8, similar to Theorem 2.19, we have
\begin{thm}
 Let $\Sigma\subset (\mathbb{H},g_L)$
  be a regular surface with finitely many boundary components $(\partial\Sigma)_i,$ $i\in\{1,\cdots,n\}$, given by Euclidean $C^2$-smooth regular and closed curves $\gamma_i:[0,2\pi]\rightarrow (\partial\Sigma)_i$.
 Suppose that the characteristic set $C(\Sigma)$ satisfies $\mathcal{H}^1(C(\Sigma))=0$ and that
$||\nabla_Hu||_H^{-1}$ is locally summable with respect to the Euclidean $2$-dimensional Hausdorff measure
near the characteristic set $C(\Sigma)$, then
\begin{equation}
\int_{\Sigma}\mathcal{K}^{\Sigma,\nabla^1,\infty}d\sigma_{\Sigma}+\sum_{i=1}^n\int_{\gamma_i}k^{\infty,\nabla^1,s}_{\gamma_i,\Sigma}d{s}=0.
\end{equation}
\end{thm}

We note that Theorem 3.9 is different from Theorem 1.1 in \cite{BTV}.

\section{ The sub-Riemannian limit and the adapted connection}
\indent We define the adapted connection $\nabla^2$ on the Heisenberg group by $\nabla^2_{X_j}X_k=0$ for any $j,k$. Then $\nabla^2$ is a metric
connection. Similar to the definition 2.3, we can define the curvature
$k^{L,\nabla^2}_{\gamma}$ associated to $\nabla^2$ of $\gamma$ at $\gamma(t)$.
We have
\begin{align}
\nabla^2_{\dot{\gamma}}\dot{\gamma}=
\ddot{\gamma}_1X_1+
\ddot{\gamma}_2X_2
+\frac{d}{dt}(\omega(\dot{\gamma}(t)))X_3.
\end{align}
Similar to Lemma 2.5, we have
\vskip 0.5 true cm
\begin{lem}
Let $\gamma:[a,b]\rightarrow (\mathbb{H},g_L)$ be a Euclidean $C^2$-smooth regular curve in the Riemannian manifold $(\mathbb{H},g_L)$. Then
\begin{align}
k^{L,\nabla^2}_{\gamma}&=\left\{\left\{
\ddot{\gamma}_1^2+
\ddot{\gamma}_2^2+L\left[\frac{d}{dt}(\omega(\dot{\gamma}(t)))\right]^2\right\}\right.\\\notag
&\cdot\left[{\dot{\gamma}_1}^2+\dot{\gamma}_2^2+L(\omega(\dot{\gamma}(t)))^2\right]^{-2}\\\notag
&-\left\{{\dot{\gamma}_1}
\ddot{\gamma}_1
+\dot{\gamma}_2
\ddot{\gamma}_2
+L\omega(\dot{\gamma}(t))\frac{d}{dt}(\omega(\dot{\gamma}(t)))\right\}^2\\\notag
&\left.\cdot\left[{\dot{\gamma}_1}^2+\dot{\gamma}_2^2+L(\omega(\dot{\gamma}(t)))^2\right]^{-3}\right\}^{\frac{1}{2}}.\\\notag
\end{align}
In particular, if $\gamma(t)$ is a horizontal point of $\gamma$,
\begin{align}
k^{L,\nabla^2}_{\gamma}&=\left\{\left\{
\ddot{\gamma}_1^2+
\ddot{\gamma}_2^2+L\left[\frac{d}{dt}(\omega(\dot{\gamma}(t)))\right]^2\right\}
\cdot\left[{\dot{\gamma}_1}^2+\dot{\gamma}_2^2\right]^{-2}\right.\\\notag
&\left.-\left\{{\dot{\gamma}_1}
\ddot{\gamma}_1
+\dot{\gamma}_2
\ddot{\gamma}_2
\right\}^2\cdot\left[{\dot{\gamma}_1}^2+\dot{\gamma}_2^2\right]^{-3}\right\}^{\frac{1}{2}}.\\\notag
\end{align}
\end{lem}
\vskip 0.5 true cm
\indent Similar to the definition 2.6, we can define the intrinsic curvature associated to the connection $\nabla^2$, $k_{\gamma}^{\infty,\nabla^2}$ of $\gamma$ at $\gamma(t)$. Similar to the lemma 2.7, we have
\vskip 0.5 true cm
\begin{lem}
Let $\gamma:[a,b]\rightarrow (\mathbb{H},g_L)$ be a Euclidean $C^2$-smooth regular curve in the Riemannian manifold $(\mathbb{H},g_L)$. Then
\begin{equation}
k_{\gamma}^{\infty,\nabla^2}=0,~~if ~~\omega(\dot{\gamma}(t))\neq 0,
\end{equation}
\begin{align}
k^{\infty,\nabla^2}_{\gamma}&=\frac{|\ddot{\gamma}_1\dot{\gamma}_2-\ddot{\gamma}_2\dot{\gamma}_1|}
{(\dot{\gamma}_1^2+\dot{\gamma}_2^2)^{\frac{3}{2}}},
~~if ~~\omega(\dot{\gamma}(t))= 0 ~~and~~\frac{d}{dt}(\omega(\dot{\gamma}(t)))=0,\\\notag
\end{align}
\begin{equation}
{\rm lim}_{L\rightarrow +\infty}\frac{k_{\gamma}^{L,\nabla^2}}{\sqrt{L}}=\frac{|\frac{d}{dt}(\omega(\dot{\gamma}(t)))|}
{\dot{\gamma}_1^2+\dot{\gamma}_2^2},~~if ~~\omega(\dot{\gamma}(t))= 0
~~and~~\frac{d}{dt}(\omega(\dot{\gamma}(t)))\neq 0.
\end{equation}
\end{lem}
Similar to (3.10), we have
\begin{align}
\nabla^{2,\Sigma}_{\dot{\gamma}}\dot{\gamma}&=
\left\{\overline{q}\ddot{\gamma}_1
-\overline{p}\ddot{\gamma}_2
\right\}e_1
+\left\{\overline{r_L}~~\overline{p}\ddot{\gamma}_1
+\overline{r_L}~~\overline{q}\ddot{\gamma}_2
-\frac{l}{l_L}L^{\frac{1}{2}}\frac{d}{dt}(\omega(\dot{\gamma}(t)))\right\}e_2.\notag
\end{align}
Similar to Definitions 2.8 and 2.9, we can define
the geodesic curvature associated to $\nabla^2$, $k^{L,\nabla^2}_{\gamma,\Sigma}$ of $\gamma$ at $\gamma(t)$
and the intrinsic geodesic curvature associated to $\nabla^2$, $k_{\gamma,\Sigma}^{\infty,\nabla^2}$ of $\gamma$ at $\gamma(t)$.
Similar to Lemma 2.10, we have
\begin{lem}
Let $\Sigma\subset(\mathbb{H},g_L)$ be a regular surface.
Let $\gamma:[a,b]\rightarrow \Sigma$ be a Euclidean $C^2$-smooth regular curve. Then
\begin{equation}
k_{\gamma,\Sigma}^{\infty,\nabla^2}=0,~~if ~~\omega(\dot{\gamma}(t))\neq 0,
\end{equation}
$$k^{\infty,\nabla^2}_{\gamma,\Sigma}=0
~~if ~~\omega(\dot{\gamma}(t))= 0, ~~and~~\frac{d}{dt}(\omega(\dot{\gamma}(t)))=0,$$
\begin{equation}
{\rm lim}_{L\rightarrow +\infty}\frac{k_{\gamma,\Sigma}^{L,\nabla^2}}{\sqrt{L}}=\frac{|\frac{d}{dt}(\omega(\dot{\gamma}(t)))|}
{\left(\overline{q}{\dot{\gamma}_1}-\overline{p}\dot{\gamma}_2\right)^2},~~if ~~\omega(\dot{\gamma}(t))= 0
~~and~~\frac{d}{dt}(\omega(\dot{\gamma}(t)))\neq 0.
\end{equation}
\end{lem}
\indent Similar to the definitions 2.11 and 2.12, we can define
the signed geodesic curvature associated to $\nabla^2$, $k^{L,\nabla^2,s}_{\gamma,\Sigma}$ of $\gamma$ at $\gamma(t)$ and
the intrinsic signed geodesic curvature associated to $\nabla^2$, $k_{\gamma,\Sigma}^{\infty,\nabla^2,s}$ of $\gamma$.
Similar to Lemma 2.13, we have
\begin{lem}
Let $\Sigma\subset(\mathbb{H},g_L)$ be a regular surface.
Let $\gamma:[a,b]\rightarrow \Sigma$ be a Euclidean $C^2$-smooth regular curve. Then
\begin{equation}
k_{\gamma,\Sigma}^{\infty,\nabla^2,s}=0,~~if ~~\omega(\dot{\gamma}(t))\neq 0,
\end{equation}
$$k^{\infty,\nabla^2,s}_{\gamma,\Sigma}=0
~~if ~~\omega(\dot{\gamma}(t))= 0, ~~and~~\frac{d}{dt}(\omega(\dot{\gamma}(t)))=0,$$
\begin{equation}
{\rm lim}_{L\rightarrow +\infty}\frac{k_{\gamma,\Sigma}^{L,\nabla^2,s}}{\sqrt{L}}=
\frac{(-\overline{q}{\dot{\gamma}_1}
+\overline{p}\dot{\gamma}_2)\frac{d}{dt}(\omega(\dot{\gamma}(t)))}
{|\overline{q}{\dot{\gamma}_1}-\overline{p}\dot{\gamma}_2|^3},
\end{equation}
$$~~if ~~\omega(\dot{\gamma}(t))= 0
~~and~~\frac{d}{dt}(\omega(\dot{\gamma}(t)))\neq 0.
$$
\end{lem}
\indent Similar to (2.26), we can define the {\it second fundamental form associated to $\nabla^2$}, $II^{\nabla^2,L}$ of the
embedding of $\Sigma$ into $(\mathbb{H},g_L)$.
Similarly to Theorem 2.14, we have
\vskip 0.5 true cm
\begin{thm} The second fundamental form $II^{\nabla^2,L}$ of the
embedding of $\Sigma$ into $(\mathbb{H},g_L)$ is given by
\begin{equation}
II^{\nabla^2,L}=\left(
  \begin{array}{cc}
   h^2_{11}, & h^2_{12}\\
   h^2_{21}, & h^2_{22}\\
  \end{array}
\right),
\end{equation}
where $$h^2_{11}= \frac{l}{l_L}[X_1(\overline{p})+X_2(\overline{q})],~~h^2_{12}=-\frac{l_L}{l}\langle e_1,\nabla_H(\overline{r_L})\rangle_L
    ,$$
    $$h^2_{21}=-\frac{l_L}{l}\langle e_1,\nabla_H(\overline{r_L})\rangle_L-\frac{\sqrt{L}}{2}
    +\frac{\sqrt{L}}{2}\frac{l^2}{l_L^2}
    -\frac{\sqrt{L}}{2}{r_L}^2,$$
   $$h^1_{22}=-\frac{l^2}{l_L^2}\langle e_2,\nabla_H(\frac{r}{l})\rangle_L+\widetilde{X_3}(\overline{r_L})
    .$$
\end{thm}

Similar to Propositions 2.15 and 2.17, we have
\begin{prop} Away from characteristic points, we have
\begin{equation}
\mathcal{H}_{\nabla^2,\infty}={\rm lim}_{L\rightarrow +\infty}\mathcal{H}_{\nabla^2,L}=X_1(\overline{p})+X_2(\overline{q}).
\end{equation}
\begin{equation}
\mathcal{K}^{\Sigma,\nabla^2}(e_1,e_2)\rightarrow \mathcal{K}^{\Sigma,\nabla^2,\infty}=0,~~{\rm as}~~L\rightarrow +\infty.
\end{equation}
\end{prop}

\vskip 1 true cm

\section{Acknowledgements}

The author was supported in part by NSFC No.11771070.

\vskip 1 true cm


\bigskip

\noindent {\footnotesize {\it Yong Wang} \\
{School of Mathematics and Statistics, Northeast Normal University, Changchun 130024, China}\\
{Email: wangy581@nenu.edu.cn}

\end{document}